\documentclass[12pt]{amsart}
\usepackage{latexsym}
\usepackage{amssymb}
\usepackage{amsxtra}
\usepackage{amsmath}
\usepackage{amsfonts}
\usepackage{amscd}
\usepackage{verbatim}
\usepackage{color}

\usepackage{amsthm,amsmath,amsfonts,amssymb,amscd,mathrsfs,graphics}
\usepackage{hyperref,supertabular}

\usepackage[top=1.5in, bottom=1.5in, left=1.5in, right=1.5in]{geometry}

\allowdisplaybreaks

\newcommand{\beq}{\begin{equation}}
\newcommand{\eeq}{\end{equation}}
\newcommand{\bqa}{\begin{eqnarray}}
\newcommand{\eqa}{\end{eqnarray}}
\newcommand{\beqa}{\begin{equation*}}
\newcommand{\ben}{\begin{eqnarray*}}
\newcommand{\eeqa}{\end{equation*}}
\newcommand{\een}{\end{eqnarray*}}

\newcommand \nc {\newcommand}

\newtheorem{theorem}{Theorem}[section]
\newtheorem{lemma}[theorem]{Lemma}
\newtheorem{proposition}[theorem]{Proposition}
\newtheorem{corollary}[theorem]{Corollary}
\newtheorem{definition}[theorem]{Definition}

\newtheorem{remark}[theorem]{Remark}

\nc \thref{Theorem \ref}
\nc \leref{Lemma \ref}
\nc \prref{Proposition \ref}
\nc \coref{Corollary \ref}
\nc \deref{Definition \ref}
\nc \exref{Example \ref}
\nc \reref{Remark \ref}

\newcommand{\leftexp}[2]{{\vphantom{#2}}^{{\rm #1}}{#2}}

\newcommand{\X}{\mathcal{X}}

\newcommand{\A}{\mathcal{A}}
\newcommand{\B}{\mathcal{B}}
\newcommand{\C}{\mathbb{C}}
\newcommand{\D}{\mathcal{D}}
\newcommand{\E}{\mathcal{E}}
\newcommand{\F}{\mathcal{F}}
\renewcommand{\H}{\mathcal{H}}

\renewcommand{\L}{\mathcal{L}}
\newcommand{\M}{\mathcal{M}}

\renewcommand{\O}{\mathcal{O}}
\renewcommand{\P}{\mathbb{P}}
\newcommand{\QQ}{\mathbb{Q}}

\renewcommand{\S}{\mathcal{S}}
\renewcommand{\SS}{\mathscr{S}}
\newcommand{\T}{\mathcal{T}}

\newcommand{\Z}{\mathbb{Z}}

\newcommand{\RR}{\mathscr{R}}
\newcommand{\HH}{\mathbb{H}}

\newcommand{\s}{\mathbf{s}}

\newcommand{\q}{\mathbf{q}}
\renewcommand{\t}{\mathbf{t}}
\newcommand{\x}{\mathbf{x}}

\def\deg{\mathop{\rm deg}\nolimits}
\def\dim{\mathop{\rm dim}\nolimits}

\def\d{\partial}

\def\iso{\cong}

\def\({\left(}
\def\){\right)}
\def\[{\left[}
\def\]{\right]}
\def\<{\left\langle}
\def\>{\right\rangle}

\def\one{{\bf 1}}

\def\LD{\langle}
\def\RD{\rangle}
\def\gl{\lambda}
\def\la{\lambda}
\def\si{\sigma}
\def\Si{\Sigma}
\def\ge{\epsilon}
\def\ga{\alpha}

\def\La{\Lambda}

\numberwithin{equation}{section}

\author{
Todor Milanov \&
Yongbin Ruan \&
Yefeng Shen}
\address{Kavli IPMU (WPI) \\ The University of Tokyo \\ Kashiwa \\ Chiba 277-8583 \\ Japan}
\email{todor.milanov@ipmu.jp}

\address{Department of Mathematics \\ University of Michigan \\ Ann Arbor \\ MI 48105 \\ USA}
\email{ruan@umich.edu}

\address{Department of Mathematics \\ University of Michigan \\ Ann Arbor \\ MI 48105 \\ USA}
\email{yfschen@umich.edu}

\begin{document}
\title[Gromov-Witten theory and cycle-valued modular forms]{Gromov-Witten theory and cycle-valued modular forms}

\maketitle
\tableofcontents
\addtocontents{toc}{\protect\setcounter{tocdepth}{1}}

\section{Introduction}

A remarkable phenomenon in Gromov-Witten theory is the appearance of (quasi) modular forms. Classically, modular forms arise as a counting function of points, i.e., zero dimensional objects. A Gromov-Witten generating function can be thought as a counting function for the virtual number of holomorphic curves, i.e., one dimensional objects. Therefore, it is natural to speculate if modular forms appear here too. One can attempt to compute them explicitly. If one is lucky enough, the answers can be organized as modular forms. Indeed, this strategy has been carried out for elliptic curves \cite{OP} and the so called reduced Gromov-Witten theory of K3-surfaces \cite{MPT}. However, we should emphasize that  both steps of the strategy are highly nontrivial. In fact, the above modularity results are some of the most sophisticated works in Gromov-Witten theory. Generally speaking, it is very difficult to compute Gromov-Witten invariants. Even if you can compute, it is not clear how to organize them into modular forms. Unlike the  case of counting points, it is impractical  to try to compute a large number of coefficients and then guess the general pattern.

In the middle of the 90's, by studying the physical B-model of Gromov-Witten theory, BCOV boldly conjectured that Gromov-Witten generating function of any Calabi-Yau manifolds are in fact {\em quasi-modular forms}. A key idea in \cite{BCOV} is that the B-model Gromov-Witten function should be modular but non-holomorphic. Furthermore, its anti-holomorphic dependence is governed by the famous {\em holomorphic anomaly equations}. During the last decade, Klemm and his collaborators have put forth a series of papers to solve the holomorphic anomaly equations \cite{ABK, MKS}. One upshot is a stunning predication of Gromov-Witten invariants of quintic 3-fold up to genus 51. Indeed, this is a great achievement since mathematicians can only compute Gromov-Witten invariants for genus zero and one. Motivated by the physical intuition, there were two independent works recently in mathematics to establish the modularity of Gromov-Witten theory rigorously for  local $\P^2$ \cite{CI2} and elliptic orbifolds $\P^1$  \cite{KS, MR}. Let's briefly describe the authors' work on the elliptic orbifolds $\P^1$. The current article can be thought as a sequel.

Let $X$ be a projective manifold and $\overline{\M}_{g,n}(X, \beta)$ be the moduli space of genus-$g$, degree-$\beta$  stable maps with $n$ markings, where  $\beta$ is a nef class in $H^2(X,\Z)$, i.e., $\beta\in{\rm NE}(X)$. Let ${\rm ev}_i$ be the evaluation map at the $i$-th marked point $p_i$ and $\psi_i\in H^{*}(\overline{\M}_{g,n})$ be the first Chern class of the cotangent line bundle at $p_i$. Choose elements $\gamma_i$ in $H^*(X, \QQ)$ with $\gamma_0=1\in H^{0}(X,\QQ)$. $\pi:\overline{\M}_{g,n}(X, \beta)\to\overline{\M}_{g,n}$ be the stabilization of the forgetful morphism. The numerical GW invariants with ancestors are defined by
\beq\label{GW inv}
\langle\tau_{\iota_1}(\gamma_{1}), \cdots, \tau_{\iota_n}(\gamma_{n})\rangle^X_{g,n,\beta}=
\int_{[\overline{\M}_{g,n}(X, \beta)]^{\rm vir}} \prod_{i=1}^{n} {\rm ev}^*_i(\gamma_i)\cup\pi^*\psi^{\iota_i}_i.
\eeq
The above invariant is zero unless
$$
\sum_{i=1}^{n} ({\rm deg}_{\C}(\gamma_i)+\iota_i)=c_1(TX)\cdot\beta+(3-\dim_{\C}X)(g-1)+n.
$$
The advantage of Calabi-Yau manifolds, such as the elliptic curve $E$, is that $c_1(TX)=0$ and hence the dimension constraint is independent of $\beta$. For the elliptic curve $E$, the degree $\beta=d\cdot\D$, where $d$ is a non-negative integer and $\D$ is a nef generator of $H^2(E,\Z)$. Then, it is natural to define
\beq\label{cor:E}
\langle\tau_{\iota_1}(\gamma_1), \cdots, \tau_{\iota_n}(\gamma_{n})\rangle^E_{g,n}(q)=
\sum_{d\geq 0}
\langle\tau_{\iota_1}(\gamma_1), \cdots, \tau_{\iota_n}(\gamma_{n})\rangle^E_{g,n,d}\,q^d,
\eeq
where $q$ is the Novikov variable that we use to keep track of the degree $\beta$. In our case, the function \eqref{cor:E} can be seen as an ancestor Gromov-Witten function along $t\cdot\D\in H^2(X,\Z)$ by setting $q=e^t$ (see Section 5).
The authors proved the modularity for   the elliptic orbifolds $\P^1$ with weights of non-trivial orbifold points are $(3,3,3), (2,4,4), (2,3,6)$. These orbifolds are the quotients of some elliptic curve $E$ by $\Z/3\Z, \Z/4\Z, \Z/6\Z$ respectively.

To state the theorem, let $\X$ be one of the three elliptic orbifolds $\P^1$. Again,  $c_1(T\X)=0$ in these
cases. We can choose elements $\gamma_i$ of $H^*_{CR}(\X)$ and  define
\beq\label{cor:X}
\langle\tau_{\iota_1}(\gamma_{i_1}), \cdots, \tau_{\iota_n}(\gamma_{i_n})\rangle^{\X}_{g,n}(q)
\eeq
similarly.
The main result of \cite{KS, MR} is the following modularity theorem.
\begin{theorem}\label{t1}
\cite{MR} Suppose that $\X$ is one of the three elliptic orbifolds $\P^1$ from above. For any multi-indices $\iota_j, i_j$, the GW invariant  \eqref{cor:X} converges to
a quasi-modular form of an appropriate weight for a finite index subgroup $\Gamma$ of $SL_2(\Z)$ under the change of variables $q=e^{2\pi i \tau/3}$, $e^{2\pi i \tau/4},$
$e^{2\pi i \tau/6},$ respectively (see \cite{MR} for the subgroup $\Gamma$ and the weights of the quasi-modular forms).
\end{theorem}
The same theorem for elliptic curves were proved ten years ago by Okounkov-Pandharipande \cite{OP}.

Recall that one can construct {\em Gromov-Witten cycles (cohomological field theories)} by a partial
integration, i.e., pushforward via the forgetfull morphism
\beq\label{GW cycle}
\Lambda_{g,n,\beta}^{X}(\gamma_{1}, \cdots, \gamma_{n})=\pi_*\big(\prod_{i=1}^{n}{\rm ev}^*_i(\gamma_i)\big)\in H^*(\overline{\M}_{g,n}, \QQ).
\eeq
The degree of the cycle is computed from the {\em dimension axiom},
\beqa
\deg_{\C}\Lambda^{X}_{g,n,\beta}(\gamma_{1}, \cdots, \gamma_{n})=(g-1)\dim_{C}(X)+\sum_{i=1}^{n}\deg_{\C}(\gamma_i)-c_1(TX)\cdot\beta.
\eeqa
The numerical Gromov-Witten invariants are obtained by
$$
\langle\tau_{\iota_1}(\gamma_1), \cdots, \tau_{\iota_c}(\gamma_n)\rangle^X_{g,n,\beta}=
\int_{\overline{\M}_{g,n}} \Lambda_{g,n,\beta}^{X}(\gamma_{1}, \cdots, \gamma_{n})\cup\prod_{i=1}^{n}\psi^{\iota_i}_{i}.
$$
Motivated by the corresponding work in number theory \cite{Z}, we want to consider the generating function of Gromov-Witten cycles
\beq\label{GW cycle func}
\Lambda_{g,n}^{X}(\gamma_{1}, \cdots, \gamma_{n})(q)=\sum_{\beta\in{\rm NE}(X)}\Lambda^{X}_{g,n,\beta}(\gamma_{1}, \cdots, \gamma_{n})\,q^{\beta}.
\eeq
We view the RHS of \eqref{GW cycle func} as a function on $q$ taking value in $H^*(\overline{\M}_{g,n}, \QQ)$. To emphasise this perspective, we sometimes refer to it as {\em cycle-valued} generating function. The main theorem of this paper is
\begin{theorem}\label{t1}
Suppose that $\X$ is one of the three elliptic orbifolds $\P^1$ with three non-trivial orbifold points;
then $\Lambda_{g,n}^{\X}(\gamma_{1}, \cdots, \gamma_{n})(q)$ converges to a  cycle-valued quasi-modular
form of an appropriate weight for a finite index subgroup $\Gamma$ of $SL_2(\Z)$ under the change of
variables $q=e^{2\pi i \tau/3}$, $e^{2\pi i \tau/4},$ $e^{2\pi i \tau/6},$ respectively.
\end{theorem}

We should mention that the above cycle-valued modularity theorem is not yet known for elliptic curve.

We obtain the modularity of numerical Gromov-Witten invariants by integrating the $\Lambda_{g,n}^{\X}(\gamma_{1}, \cdots, \gamma_{n})$ with
psi-classes over the fundamental cycle $[\overline{\M}_{g,n}]$. On the other hand, we can also use other interesting classes of $\overline{\M}_{g,n}$ such as $\kappa_i$'s or Hodge class $\lambda_i$'s.

Suppose that $P$ is a polynomial of $\psi_i, \kappa_i, \lambda_i$. We define a {\em generalized numerical Gromov-Witten invariants}
$$\langle\gamma_{1}, \cdots, \gamma_{n}; P\rangle^X_{g,n,\beta}=\int_{\overline{\M}_{g,n}} P\cup\Lambda_{g,n,\beta}^{X}(\gamma_{1}, \cdots, \gamma_{n})
$$
and its generating function
$$\langle\gamma_{1}, \cdots, \gamma_{n}; P\rangle^X_{g,n}(q)=\sum_{\beta\in{\rm NE}(X)}\langle\gamma_{1}, \cdots, \gamma_{n}; P\rangle^X_{g,n,\beta}\,q^{\beta}.$$
Here, we set it to be zero if the dimension constraint are not satisfied.
\begin{corollary}
Suppose that $\X$ is one of the above three elliptic orbifolds $\P^1$. Then, the above generalized numerical Gromov-Witten generating functions are quasi-modular forms for the same modular group and weights given by the main theorem.
\end{corollary}

Recall that the proof of the numerical version consists of two steps. The first step is to construct a higher genus B-model theory (modulo an extension problem) and prove its modularity. Then, the second step is to prove mirror theorems to match it with a Gromov-Witten theory which will solve the extension property as well as inducing the modularity for a Gromov-Witten theory. In this paper, we follow the same outline, i.e., our strategy can be carried out on the cycle level. Our main new ingredient is Teleman's reconstruction theorem \cite{Te}.

The paper is organized as follows. In Section \ref{sec:2}, we will review the action of upper-triangular symplectic operators on a cohomological field theory, which will be the main tool of the paper. In Section \ref{sec:3}, we review the construction of global Frobenius manifold structures from \cite{MR}. Using it, we can define Givental B-model cohomological field theory as indicated by Telemann \cite{Te}. In Section \ref{sec:4}, we calculate the action of the monodromy group on the Givental's B-model cohomological field theory and prove the (quasi-)modularity. Finally, in  Section \ref{sec:5}, we prove the mirror theorems on the cycle level. Here, the original $g$-reduction argument does not apply. We replace it by Teleman's reconstruction theorem \cite{Te}.

\subsection{Acknowledgements}

The work of the first author is supported by Grant-In-Aid and by the World Premier International Research Center Initiative (WPI Initiative), MEXT, Japan. The second author is partially supported by a NSF grant. The
second and third authors would like to thank Hiroshi Iritani for interesting discussions on the convergence
of Gromov-Witten theory. The third author would like to thank Emily Clader, Nathan Priddis and Mark Shoemaker for helpful discussions on Givental's theory. Finally, three of us would like to thank IPMU for hospitality where the part of this work is carried out. We thank Arthur Greenspoon for editorial assistance.

\section{Cohomological field theory and quantization}\label{sec:2}

The quantization formalism in Gromov-Witten theory was introduced by Givental in \cite{G} and then revisited by Teleman at the cohomological field theory level in \cite{Te}. The latter will be used in this article. For the readers' convenience, we give a brief introduction here. From now on, let
$\pi_{g,n,k}: \overline{\M}_{g, n+k}\rightarrow \overline{\M}_{g,n}$
be the stabilization of the morphism that forgets the last $k$ marked points. For simplicity, we will omit the subscripts if they are indicated in the context.
\subsection{Cohomological field theories}\label{sec:2.1}
Let $H$ be a vector space of dimension $N$ with a unit {\bf 1} and a non-degenerate paring $\eta$. Without loss of generality, we always fix a basis of $H$, say $\mathscr{S}:=\{\d_i, i=0,\cdots,N-1\}$, and we set $\d_0={\bf 1}$. Let $\{\d^j\}$ be the dual basis in the dual space $H^{\vee}$,
(i.e., $\eta(\d_i,\d^j)=\delta_i^j$).
A \emph{Cohomological field theory} (or CohFT) is a set of multi-linear maps $\La=\{\La_{g,n}\}$,
with $$\Lambda_{g,n}: H^{\otimes n}\longrightarrow H^*(\overline{\M}_{g,n}, \C),$$
or equivalently,
$$\Lambda_{g,n}\in H^*(\overline{\M}_{g,n}, \C)\otimes (H^{\vee})^{\otimes n},$$
defined for each stable genus $g$ curve with $n$ marked points,i.e., $2g-2+n>0$.
Furthermore, $\La$ satisfies a set of axioms (CohFT axioms) described below:
\begin{enumerate}
\item ($S_n$-invariance) For any $\sigma\in S_n$, and $\gamma_1,\cdots,\gamma_n\in H$; then
$$\Lambda_{g,n}(\gamma_{\sigma(1)}, \cdots, \gamma_{\sigma(n)})=\La_{g,n}(\gamma_1,\cdots,\gamma_n).$$
\item (Gluing tree)
Let
$$\rho_{tree}:\overline{\M}_{g_1,n_1+1}\times\overline{\M}_{g_2,n_2+1}\to\overline{\M}_{g,n}$$
where $g=g_1+g_2, n=n_1+n_2$, be the morphism induced from gluing the last marked point of the first curve and the first marked point of the second curve; then
\begin{align*}
\rho_{tree}^*&\big(\Lambda_{g,n}(\gamma_1,\cdots,\gamma_n)\big)\\
&=\sum_{\alpha,\beta\in\SS}\Lambda_{g_1,n_1+1}(\gamma_1,\cdots,\gamma_{n_1},\alpha)\eta^{\alpha,\beta}\Lambda_{g_2,n_2+1}(\beta,\gamma_{n_1+1},\cdots,\gamma_n).
\end{align*}
Here $\big(\eta^{\alpha,\beta}\big)_{N\times N}$ is the inverse matrix of $\big(\eta(\alpha,\beta)\big)_{N\times N}$.
\item (Gluing loop)
Let
$$\rho_{loop}:\overline{\M}_{g-1,n+2}\to\overline{\M}_{g,n},$$
be the morphism induced from gluing the last two marked points; then
\beqa
\rho_{loop}^*\big(\Lambda_{g,n}(\gamma_1,\cdots,\gamma_n)\big)=\sum_{\alpha,\beta\in\SS}\Lambda_{g-1,n+2}(\gamma_1,\cdots,\gamma_n,\alpha,\beta)\eta^{\alpha,\beta}.
\eeqa
\item (Pairing)
$$\int_{\overline{\M}_{0,3}}\Lambda_{0,3}({\bf 1}, \gamma_1, \gamma_2)=\eta(\gamma_1, \gamma_2).$$
\end{enumerate}
If in addition the following axiom holds
\begin{itemize}
\item[(v)] (Flat identity) Let
  $\pi:\overline{\M}_{g,n+1}\to\overline{\M}_{g,n}$ be the forgetful morphism; then
$$\Lambda_{g, n+1}(\gamma_1,\cdots,\gamma_n, {\bf 1})=\pi^* \Lambda_{g,n}(\gamma_1,\cdots,\gamma_n).$$
\end{itemize}
then we say that $\La$ is a CohFT with a {\em flat identity}.

Note that $\La_{0,3}$ will induce a Frobenius multiplication $\bullet$ on $(H,\eta)$, defined by
\beq\label{frobenius}
\eta(\alpha\bullet\beta,\gamma)=\int_{\overline{\M}_{0,3}}\Lambda_{0,3}(\alpha,\beta,\gamma);
\eeq
We refer to $(H,\eta,\bullet)$ as the Frobenius algebra underlying $\La$, or simply as the
{\em state space} of $\La$.  The CohFT is called {\em semisimple} if the underlying Frobenius algebra is semisimple.

\subsection{Examples of CohFTs}
Let $\C^N$ be the complex vector space equipped with the standard bi-linear pairing: $(e_i,e_j)=\delta_{i,j}$.
Let $\Delta=(\Delta_1,\cdots,\Delta_N)$ be a sequence of {\em non-zero} complex numbers. The following definition
\beq\label{trivial-cohft}
I^{N,\Delta}_{g,n}(e_{i_1},\dots,e_{i_n}):=
\begin{cases}
\Delta_{i}^{g-1+\frac{n}{2}} & \mbox{if  } i=i_1=i_2=\cdots=i_n,\\
0 & \mbox{otherwise},
\end{cases}
\eeq
induces a CohFT on $\C^N$ which we call a rank $N$ {\em trivial} CohFT. The Frobenius algebra underlying
$I^{N,\Delta}$ will be denoted by $(\C^N,\Delta)$. Note that the Frobenius multiplication is
given by
$$
e_i\bullet e_j=\delta_{ij} \, \sqrt{\Delta_i} \, e_i.
$$

Another famous example comes from Gromov-Witten theory (cf. \cite{KM, CR}). Let $X$ be a projective variety (or orbifold), let $H$ be its cohomology $H^*(X)$ (or Chen-Ruan cohomology $H^*_{\rm CR}(X)$), $\eta$ be the Poincar\'e pairing. Then $\La_{g,n}^X(q)$ defined in \eqref{GW cycle func} gives a CohFT for $q=0$.
The above axioms make sense for cohomology classes $\Lambda^{X}_{g,n}(q)$ that have coefficients in some ring of formal power series. In such a case we say that we have a {\em formal} cohomological field theory. A priori, the CohFT in \eqref{GW cycle func} is only formal.

\subsection{Givental's formalism}\label{sec:2.2}


Following Givental, we introduce the vector space $\H=H((z))$ of
formal Laurent series in $z^{-1}$.
Furthermore, $\H$ is  equipped with the following symplectic structure $\Omega$:
\ben
\Omega(f(z),g(z))= {\rm res}_{z=0} (f(-z),g(z))dz, \quad f(z),g(z)\in\H,
\een
where for brevity we put $(a,b)=\eta(a,b)$ for $a,b\in H$.  Note that $\H$ has a polarization
$$\H=\H_{+}\oplus\H_{-}$$
with $\H_{+}=H[z]$ and $\H_{-}=z^{-1}H[[z^{-1}]]$, which allows us to identify  $\H\cong T^*\H_+$.
We fix a Darboux coordinate system $q_k^i,p_{l,j}$ for $\H$ via
$$f(z)=\sum_{k=0}^{\infty}\sum_{i=0}^{N-1} q^{i}_k\, \d_i
z^{k}+\sum_{l=0}^{\infty}\sum_{j=0}^{N-1} p_{l,j}\,\d^j(-z)^{-l-1}\in\H,$$
For convenience, we put
\beq\label{q-variable}
\q_k:=(q_k^{1},\cdots,q_k^{N}) \hspace{10 mm}\text{and}\hspace{10 mm} \q:=(\q_0,\q_1,\cdots).
\eeq
In this paper, we focus on the subgroup $\L^{(2)}{\rm GL}(H)$ of the
loop group $\mathcal{L}{\rm GL}(H)$  consisting of symplectomorphisms
$T:\H\to\H$. Note that such symplectomorphisms are defined by the
following equation:
$$^*T(-z)T(z)=\rm{Id},$$
where $^*T$ is the adjoint operator with respect to the bi-linear pairing $\eta$, i.e.,
$$(^*T f, g)=(f, T g).$$
We will allow symplectomorphism $E$ of the following form:
\beqa
E:={\rm Id}+E_1 z+E_2 z^2+\cdots \in{\rm End}(H)[[z]].
\eeqa
They form a group which we denote by $\L^{(2)}_+{\rm GL}(H)$ and we
refer to its elements as \emph{upper-triangular} transformations.

Next, we want to define the quantization $\widehat{E}$. Note that $A=\log E$ is a well-defined
{\em infinitesimal symplectomorphism}, i.e., $^*A=-A$.
For any infinitesimal symplectomorphism $A$, we can associate  a {\em quadratic Hamiltonian $h_{A}$} on $\H$,
\beq
h_A(f)=\frac{1}{2}\Omega(Af,f).
\eeq
The quadratic Hamiltonians are quantized by the rules:
\beq\label{eq:quadratic}
(p_{k,i}p_{l,j})\sphat = \hbar\frac{\d^2}{\d q_k^i\d q_l^j},\quad
(p_{k,i}q_l^j)\sphat = (q_l^jp_{k,i})\sphat = q_l^j\frac{\d}{\d q_k^i},\quad
(q_k^iq_l^j)\sphat =\frac{q_k^iq_l^j}{\hbar}.
\eeq
The quantization of $E$ is defined by
\beqa
\widehat{E}=e^{\widehat{A}}:=e^{\widehat{h_A}}.
\eeqa

For an upper-triangular symplectomorphism $E$, there is an explicit formula for the quantization $\widehat{E}$. Put
$$\q(z)=\sum_{k=0}^{\infty}\sum_{i=0}^{N-1}q^{i}_k\, \d_i
z^{k}\in H[[z]].$$
Denote the {\em dilaton shift} by $\widetilde{\q}(z)=\q(z)+{\bf 1}z$, i.e., $\widetilde{q}^{i}_{k}=q^{i}_{k}+\delta^{1}_{k}\delta^{i}_{0}$.
Recall that {\em the ancestor GW potential} of $X$ is
\beq
\A^{X}(\hbar,\q(z))
:=\exp\Big(\sum_{g,n}\sum_{\beta\in{\rm NE}(X)}\sum_{\iota_i,k_i=0}^{\infty}\frac{\hbar^{g-1}\LD\tau_{\iota_1}\partial_{k_1},\cdots,\tau_{\iota_n}\partial_{k_n}\RD_{g,n,\beta}^{X}\,q^{\beta}}{n!}\prod_{i=1}^{n}\widetilde{q}_{k_i}^{\iota_i}\Big).
\eeq
$\A^{X}(\hbar,\q(z))$ belongs to a Fock space $\C[[\q_0,\widetilde{\q}_1,\q_2,\cdots]]$.
The action of the quantization operator $\widehat{E}$, whenever it makes sense, is given by the following formula:
\beq\label{R:fock}
\widehat{E}\big(\A^{X}(\hbar,\q(z))\big)= \left.\big(e^{W_E}\A^{X}(\hbar,\q(z))\big)\right|_{\q\mapsto E^{-1}\q},
\eeq
where $E^{-1}\q$ is the change of $\q$-coordinate
$$(E^{-1}\q)_k^i=\sum_{l=0}^k\sum_{j=0}^{N-1}(E^{-1})_l^{ji}q_{k-l}^j.$$
And $W_E$ is the quadratic differential operator
\beq\label{eq:quad}
W_E:=\frac{\hbar}{2}\sum_{k,l=0}^{\infty}\sum_{i,j=0}^{N-1}\( \d^i,V_{kl}(\d^j)\)\frac{\d^2}{\d  q_k^i \d q_l^j},
\eeq
whose coefficients $V_{kl}\in{\rm End}(H)$ are given by
\beq\label{V:fock}
\sum_{k,l\geq0}V_{kl}(-z)^k(-w)^l = \frac{E^*(z)E(w)-{\rm Id}}{z+w}.
\eeq
\begin{remark}
Givental also considered the quantization of a general symplectomorphism of the form $e^A$.
For example, $A$ could be lower triangular in the sense containing the negative power of $z$.
The lower triangular one can not be lift to cycle level. Hence, it will not be considered here.
\end{remark}

\subsection{Cycle-valued Quantization}\label{sec:2.3}
Teleman \cite{Te} was able to lift the quantization of an upper triangular symplectic transformation to the level of cohomological field theory.
Let us describe his construction. According to formula \eqref{R:fock},
the action of $\widehat{E}$ is a composition of two operations: exponential of the Laplace type operator \eqref{eq:quad}
followed by the {\em coordinate change} $\q\mapsto E^{-1}\q$.

\subsubsection{Coordinate Change}\label{sec:2.3.1}
Let $\La_{g,n}$ be any multi-linear function on $H^{\otimes n}$ with values in the
cohomology ring of $\overline{\M}_{g,n}$.
We can extend $\Lambda_{g,n}$ from $H^{\otimes n}$ to $\H^{\otimes n}_+$ uniquely so that multiplication by $z$
is compatible with the multiplication by psi-classes, i.e.,,
\beq\label{z-psi}
\Lambda_{g,n}(\sum_{i\geq 0} \gamma_1 z^i, \cdots)=\sum_{i\geq 0}\Lambda_{g,n}(\gamma_1, \cdots)\psi^i_1.
\eeq
Given an isomorphis of $\C[z]$-modules
$$
\Phi(z):H_1[[z]]\to H_{2}[[z]],
$$
we define
$$
(\Phi(z)\circ\La)_{g,n}(\gamma_1,\cdots,\gamma_n)
=\La_{g,n}(\Phi(z)^{-1}(\gamma_1),\cdots,\Phi(z)^{-1}(\gamma_n))\in H^*(\overline{\M}_{g,n},\C).
$$
Note that even if $\Lambda$ is a CohFT, $\Phi(z)\circ\Lambda$ might fail to be a CohFT.

\subsubsection{Feynman type sum}\label{sec:2.3.2}

The action of the exponential of the Laplacian \eqref{eq:quad} can be described in terms of sum over graphs.
Let us explain this in some more details. For a given graph $\Gamma$ let us denote by $V(\Gamma)$ the set
of vertices, $E(\Gamma)$ the set of edges, and by $T(V)$ the set of tails. For a fixed vertex $v\in V(\Gamma)$
we denote by $E_v(\Gamma)$ and $T_v(\Gamma)$ respectively the set of edges and tails incident with $v$.
The graph is {\em decorated} in the following way: each vertex $v$ is assigned a non-negative number  $g_v$
called genus of $v$; there is a bijection $t\mapsto m(t)$ between the set of tails and the set of integers
$\{1,2,\dots,{\rm Card}(T(\Gamma))\}$, and finally every flag $(v,e)$ (i.e., a pair consists a vertex and an incident edge)
is decorated with a vector $z^k \d^i$ ($k\geq 0$).

Furthermore, for a given edge $e$ we define a {\em propagator} $V_e$ as follows. Let $v'$, $v''$ be the two vertexes
incident with $e$ and let $z^{k'}\d^{i'}$ and $z^{k''}\d^{i''}$ be the labels respectively of the flags $(v',e)$ and $(v'',e)$; then
we define
\ben
V_e = \(\d^{i'},V_{k',k''}\d^{i''}\).
\een
Note that since $^*V_{k',k''}=V_{k'',k'}$ the definition of $V_e$ is independent of the orientation of the edge $e$.
For every vertex $v$ we define the differential operator
\ben
D_\q^v = \prod_{e\in E_v(\Gamma)} \, {\d}/{\d q_{k(e)}^{i(e)}} ,
\een
where $z^{k(e)}\d^{(i(e)}$ is the label of the flag $(v,e)$.

Given any formal function
$\A(\hbar;\q)=\exp\Big(\sum\hbar^{g-1} \F^{(g)}(\q)\Big)$
we have
\beq\label{feynman}
e^{W_E}\ \A(\hbar;\q) = \exp\Big( \sum_{\Gamma} \frac{1}{|{\rm Aut}(\Gamma)|}\,
\prod_{e\in E(\Gamma)}\, V_e \prod_{v\in V(\Gamma)}\, D_\q^v \F^{(g_v)}(\q)\ \Big),
\eeq
where the sum is over all connected decorated graphs $\Gamma$ and $|{\rm Aut}(\Gamma)|$ is the number of automorphisms
of $\Gamma$ compatible with the decoration.

Motivated by formula \eqref{feynman} we define
\beq\label{do:CohFT}
(e^{W_E}\circ \Lambda)_{g,n}(\gamma_1\otimes\cdots\otimes \gamma_n)
\eeq
by the following formula:
\ben
\sum_{\Gamma} \frac{1}{|{\rm Aut}(\Gamma)|}\,
\prod_{e\in E(\Gamma)}\, V_e \prod_{v\in V(\Gamma)}\, \La_{g_v, r_v+n_v}\Big(
\otimes_{e\in E_v(\Gamma)} \d_{i(e)} \psi^{k(e)}\otimes_{t\in T_v(\Gamma)} \gamma_{m(t)}\Big),
\een
where $r_v={\rm Card}(E_v(\Gamma))$, $n_v={\rm Card}(T_v(\Gamma))$, and the sum is over all connected, decorated, genus-g
graphs $\Gamma$ with $n$ tails. Note that this definition is compatible with \eqref{feynman} in a sense that the potential of
the multi-linear maps \eqref{do:CohFT} coincides with \eqref{feynman}.

For an upper-triangular symplectic transformation $E$, we define
\beq\label{Quant-E}
\widehat{E}\circ\Lambda:=E\circ(e^{W_E}\circ\Lambda).
\eeq
Using induction on the number of nodes, it is not hard to check that $\widehat{E}\circ\La$ is a CohFT (see \cite{Te}).

\subsubsection{Classification of semi-simple CohFT}\label{sec:2.3.3}
Let $(H,\eta,\bullet)$ be a semi-simple Frobenius algebra. We pick an orthonormal basis $\{e_i\}$ of $H$, which allows us
to identify  $(H,\eta,\bullet)$ with the Frobenius algebra of a trivial CohFT, i.e.,, the state space of
$I^{N,\Delta}$ for a particular $\Delta$ (see \eqref{trivial-cohft}). In this
section we would like to recall the classification of all CohFTs whose state space is $(H,\eta,\bullet)$.

The space of such CohFTs admits the action \eqref{Quant-E} of the group $\L^{(2)}_+{\rm GL} (H)$. Note that this action does not change the Frobenius multiplication on $H$. On the other hand, the Abelian group $z^2H[z]$ (with group operation addition) acts on the space of CohFTs via translations. Namely, given $a(z)\in z^2H[z]$, we define
\ben
(T_{a(z)}\circ\Lambda)_{g,n}
(\gamma_1,\cdots,\gamma_n)
\een
by the fomula
\beq\label{eq:trans}
\sum_{k\geq0}\frac{(-1)^k}{k!}
\pi_*\big(\Lambda_{g,n+k}(\gamma_1,\cdots,\gamma_n,a(z),\cdots,a(z))\big),
\eeq
where for each $k$, the map $\pi$ in the $k$-th summand is the map forgetting the last $k$ marked points.
This action also preserves the Frobenius multiplication. Moreover, the following formula holds:
\beq\label{trans:twist}
T_{a(z)} \circ \widehat{E}\circ T^{-1}_{a(z)} = \widehat{E}\circ T_{a(z)-E^{-1}a(z)},
\eeq
i.e.,, we have an action of the group $z^2H[z]\rtimes \L^{(2)}_{+}{\rm GL}(H)$ on the set of CohFTs with state space
$(H,\eta,\bullet)$.
According to Teleman (see \cite{Te}, Theorem 2)
\begin{theorem}(\cite{Te})\label{teleman2}
The orbit of the group $z^2H[z]\rtimes \L^{(2)}_{+}{\rm GL}(H)$ containing
$I^{N,\Delta}$ consists of all CohFTs whose underlying Frobenius algebra is $(H,\eta, \bullet)$.
\end{theorem}

Let $a(z)\in zH[z]$ be arbitrary. Although the translation $T_{a(z)}$ is singular
and will be not well defined after replacing multiplication by $z$ in terms of multiplication by psi-classes, the RHS of formula \eqref{trans:twist} always
makes sense since $a(z)-E^{-1}a(z)\in z^2H[z]$. Therefore, we can define
the following subgroup of $z^2H[z]\rtimes \L^{(2)}_{+}{\rm
  GL}(H)$:
\ben
\L^{(2)}_{a(z)}{\rm GL}(H) = T_{a(z)}\circ \L^{(2)}_{+}{\rm
  GL}(H)\circ T^{-1}_{a(z)} .
\een
The following fact follows easily from Theorem \ref{teleman2}
\begin{corollary}
Let $a(z)={\bf 1}\,z$; then the orbit of the subgroup $\L^{(2)}_{a(z)}{\rm GL}(H)$ containing $I^{N,\Delta}$
consists of all CohFTs with a flat identity whose underlying Frobenius algebra is $(H,\eta,\bullet)$.
\end{corollary}

\subsubsection{Higher-genus reconstruction}
For $\t=\sum_{i=0}^{N-1}t_i\d_i\in H,$
we define a {\em translation} operator $T_{\t}$ acting
on a CohFT $\Lambda$ by
\beq\label{eq:trans}
(T_{\t}\circ\Lambda)_{g,n}
(\gamma_1,\cdots,\gamma_n) :=
\sum_{k\geq0}\frac{(-1)^k}{k!}\pi_*\big(\Lambda_{g,n+k}(\gamma_1,\cdots,\gamma_n,\t,\cdots,\t)\big).
\eeq
For brevity we put $_{\t}\La :=T_{\t}\circ\Lambda$.
According to Teleman (see \cite{Te}, Proposition 7.1), ${}_{\t}\Lambda$ is a {\em formal} CohFT, i.e.,,
\beqa
({}_{\t}\Lambda)_{g,n}\in \big(H^*(\overline{\M}_{g,n},\C)\otimes\C[[\t]]\big)\otimes (H^{\vee})^{\otimes n}, \quad \C[[\t]]:=\C[[t_0,\cdots,t_{N-1}]].
\eeqa
It induces a ring structure on $H$ with the multiplication
$\bigstar_{\t}$ defined by
\beq
\gamma_1\bigstar_{\t}\gamma_2=\sum_{\alpha,\beta\in\mathscr{S}}\int_{\overline{\M}_{0,3}}({}_{\t}\Lambda)_{0,3}(\gamma_1,\gamma_2,\alpha)\,\eta^{\alpha,\beta}\,\beta,\quad \gamma_1,\gamma_2\in H.
\eeq
Let us assume that the vector space $H$ is graded and that $\{\d_i\}$ is a homogeneous basis with ${\rm deg}(\d_i)=1-d_i$.
We further assume that we are given an {\em Euler vector field} of the form
\ben
\E=\sum_{i=0}^{N-1} \, d_i\,t_i\, \frac{\d}{\d t_i} + \sum_{j:d_j=0}\ r_j \, \d_j,
\een
where $r_j$ are some constants, so that the ring of formal power series $\C[[\t]]$ is a graded ring: an element $f(\t)$
is homogeneous of degree $d_f$ iff $E(f(\t)) = d_f \, f(\t)$.
The CohFT is called {\em homogeneous} of conformal weight $d$ if $H$ is graded, there exists an Euler vector field, and
the maps:
\ben
{}_\t\La_{g,n}: H^{\otimes n} \rightarrow H^*(\overline{\M}_{g,n};\C)\otimes \C[[\t]]
\een
are homogeneous of weight $d(g-1)+n$. Here the source of $ {}_\t\La_{g,n}$ inherits the grading from $H$, the first tensor factor
in the target is graded by {\em halfing} the degree of a cohomology class, and $\C[[\t]]$ is graded by $E$.

Let us assume further that the CohFT $\La$ is \emph{generically semisimple}, i.e.,, the ring
structure $\bigstar_{\t}$ is semisimple for generic $\t$. We denote by $u_i(\t)$ the corresponding canonical coordinates,
so that the map
\ben
\Psi(\t):\C^N \rightarrow H,\quad e_i\mapsto \sqrt{\Delta_i(\t)} \d/\d u_i(\t)
\een
identifies the Frobenius algebras $(H,\eta, \bigstar_\t)$ and $(\C^N,\Delta(\t))$, where
\ben
\Delta(\t):=(\Delta_1(\t),\dots,\Delta_N(\t)).
\een
Let $U(\t)$ be the diagonal matrix with entries
$u_i(\t), 0\leq i\leq N-1$. Following Givental we define an upper-triangular
symplectic transformation $R(\t)$, such that the formal asymptotical series $\Psi(\t)R(\t)e^{U(\t)}$ is a solution
to the differential equations \eqref{frob_eq1} and \eqref{frob_eq2}. In fact, these equations determine $R(\t)$
uniquely in terms of  $\Psi(\t)$ and $U(\t)$ (see \cite{G2}). According to Teleman (see \cite{Te}, Theorem 1) we have the
following higher-genus reconstruction result.
\begin{theorem}(\cite{Te})\label{hg-rec}
If $\La$ is a homogeneous CohFT with flat identity and $\bigstar_\t$ is (formal) semi-simple; then
\ben
{}_\t\La = \widehat{\Psi}(\t)\circ (T_{z} \circ \widehat{R}(\t)\circ T^{-1}_{z} )\circ I^{N,\Delta(\t)},
\een
where $T_z:=T_{\one\,z}$.
\end{theorem}
Let us finish this section by drawing an important corrolary from the above theorem.
The {\em ancestor potential} of a (singular) CohFT $\La$ is a function of $a(z)=\sum_{i\geq0}a_iz^i\in H[[z]]$, defined by
\beq\label{cohft-potential}
\A\big(\La,\hbar,a(z)\big)=
\exp\Big(\sum_{g,n}\frac{\hbar^{g-1}}{n!}\int_{\overline{\M}_{g,n}}\La_{g,n}(a(z),\cdots,a(z))\Big).
\eeq
Note that the translation $T_z^{-1}$ induces the so called {\em dilaton shift}, i.e, $\q(z)\mapsto\widetilde{\q}(z)=: \q(z)+{\bf 1}z$,
\beq
\A\big(T_{z}^{-1}\circ\La,\hbar,\q(z)\big)=\A\big(\La,\hbar,\widetilde{\q}(z)\big).
\eeq
Furthermore, following Givental, the {\em formal} ancestor potential
\ben
\A^{{formal}}(\t)(\hbar,\q(z))
\een
of (the germ of) the Frobenius structure $(H,\eta,\bigstar_\t)$ is defined by
\beq\label{formal-potential}
\widehat{\Psi}(\t)\ \widehat{R}(\t)\
e^{ (  U(\t)/z  )\sphat }\
\prod_{i=1}^{\mu}
\A^{\rm pt}\big(\hbar\Delta_i(\t),{}^{i}\widetilde{\q}(z)\sqrt{\Delta_i(\t)}\big),
\eeq
where $\widehat{\Psi}$ means change of the variables $\q(z)\mapsto \Psi^{-1}(\t)\,\q(z)$ and $\A^{\rm pt}$
is the total ancestor potential of the CohFT $I^{N=1,\Delta=1}$.
\begin{corollary}
Under the same assumption as in Theorem \ref{hg-rec} the following formula holds:
$$
\A({}_{\t}\La,\hbar,\widetilde{\q}(z))=\A^{{formal}}(\t)(\hbar,\q(z)).$$
\end{corollary}
Finally, let us point out that if $\La^X$ is the CohFT induced from the Gromov-Witten theory of $X$; then
\beq\label{potential=}
\A^{X}(\hbar,\q(z)):=\A(\La^{X},\hbar,\widetilde{\q}(z)),
\eeq
coincides with the so called {\em total ancestor potential} of $X$. In particular, $\A^{\rm pt}$ in formula
\eqref{formal-potential} is the total ancestor potential of a point.

\section{Global Frobenius manifolds for simple elliptic singularities}\label{sec:3}

In this section, we review the construction of global Frobenius manifolds for simple elliptic singularities. We follow \cite{MR}.

\subsection{Saito theory}
K. Saito's theory of primitive forms \cite{S} yields a certain flat structure on the space of miniversal deformations of a singularity, which is known
to be a Frobenius structure cf. \cite{He,SaT}. We refer to it as the Saito's Frobenius manifold structure. Let us recall the general set up.

Recall (see \cite{AGV}) the action of the group of germs of holomorphic changes of the coordinates $(\C^N,0)\to (\C^N,0)$ on the space of all germs at $0$ of holomorphic functions. Let ${\bf x}=(x_1,\cdots,x_N)\in\C^N$. Given a holomorphic germ $f({\bf x})$ with an isolated critical point at ${\bf x}=0$, we say that the family of functions $F(\s,{\bf x})$ is a {\em miniversal deformation} of $f$ if it is transversal to the orbit of $f$. One way to construct a miniversal deformation is to choose a $\C$-linear basis $\{\phi_i({\bf x})\}_{i=0}^{\mu-1}$ in the Jacobi algebra $\O_{\C^N,0}/\langle \d_{x_1}f,\cdots,\d_{x_{N}}f\rangle$. Here $\mu$ is the rank of the Jacobi algebra as a vector space, also known as the {\em Milnor number} or the {\em multiplicity} of the critical point. Then the following family provides a miniversal deformation:
\beq\label{eq:miniversal}
F(\s,{\bf x})=f({\bf x})+\sum_{i=0}^{\mu-1} s_i\phi_i({\bf x}),\quad \s=(s_0,s_1,\dots,s_{\mu-1})\in \S,
\eeq
where $\S\subset \C^\mu$ is a small ball around $0\in \C^\mu$. The domain of the function $F(s,\, )$ is chosen uniformly for $s\in \S$ to be a certain open
naighbourhood of $0\in \C^N$, such that its boundary satisfies certain transversality conditions (see \cite{AGV}). Slightly abusing the notation we write $\C^N$,
but we really mean an appropriate chosen open neighborhood of $0.$ This should not cause any confusion. Moreover, we will apply Saito's theory only to singularities
for which this special domain does coincide with $\C^N$.

Put $X=\S\times \C^N$ and let $C\subset X$ be the critical set of $F$, that is the support of the sheaf
\ben
\O_C:=\O_X/\langle \d_{x_0} F,\cdots,\d_{x_{N-1}}F\rangle.
\een
We have the following maps:
\ben
\begin{CD}
 \S\times \C^N & & \\
@V{\varphi}VV\searrow &  \\
\S\times \C &@>>{p}> & \S
\end{CD}
\qquad
\begin{tabular}{rl}
&$\varphi(\s,{\bf x})= (\s,F(\s,{\bf x}))$,\\
&\\
&$p(\s,\gl)=\s,$
\end{tabular}
\een
The map $\d/\d s_i\mapsto \d F/\d s_i$ induces an isomorphism between the sheaf $\T_\S$ of holomorphic vector fields on $\S$ and $q_*\O_{C}$, where $q=p\circ\varphi.$
In particular, for any $\s\in\S,$ the tangent space $T_{\s}\S$ is equipped with an associative commutative multiplication $\bullet_{\s}$ depending holomorphically on $\bf{s}\in \S$. In addition, if we have a volume form $\omega=g(\s,{\bf x})d^N{\bf x},$ where $d^N{\bf x}=dx_1\wedge\cdots\wedge dx_{N}$ is the standard volume form; then $q_*\O_C$ (hence $\T_\S$ as well) is equipped with the {\em residue pairing}:
\beq\label{res:pairing}
\LD\psi_1,\psi_2\RD= \frac{1}{(2\pi i)^N} \ \int_{\Gamma_\ge} \frac{\psi_1({\bf s,y})\psi_2({\bf s,y})}{F_{y_1}\cdots F_{y_N}}\, \omega,
\eeq
where ${\bf y}=(y_1,\cdots,y_N)$ are unimodular coordinates for the volume form, i.e., $\omega=d^N{\bf y}$, and $\Gamma_\ge$ is a real $N$-dimensional cycle supported on $|F_{x_i}|=\ge$ for $1\leq i\leq N.$

Given a holomorphic function $f$ on $\C^N$ and a real number $m$ we define
\ben
{\rm Re}^m_{f}(\C^N) := \Big\{x\in \C^N\ :\ {\rm Re}(f(x))\leq m\Big\}.
\een
Let
\beq\label{osc_integral}
J_{\A}(\s,z) = (-2\pi z)^{-N/2} \, zd_\S \, \int_{\A} e^{F({\bf s,x})/z}\omega,
\eeq
where $d_\S$ is the de Rham differential on $\S$ and $\A$ is a semi-infinite cycle from
\beq\label{cycle}
\lim_{ \longleftarrow } H_N(\C^N,{\rm Re}_{F(\s,\cdot)/z}^{-m}(\C^N);\C)\iso \C^\mu.
\eeq
By definition, the oscillatory integrals $J_\A$ are sections of the cotangent sheaf $\T_\S^*$.
According to Saito's theory of primitive forms \cite{S}, there exists a volume form $\omega$ such that the residue pairing is flat and the oscillatory integrals satisfy a system of differential equations, which in flat-homogeneous coordinates ${\bf t}=(t_0,\dots,t_{\mu-1})$ have the form
\beq\label{frob_eq1}
z\d_i J_\A({\bf t},z) = \d_i \bullet_{\bf t} J_\A({\bf t},z),
\eeq
where $ \d_i:=\d/\d t_i\ (0\leq i\leq\mu-1)$ and the multiplication is defined by identifying vectors and covectors via the residue pairing. Due to homogeneity the integrals satisfy a differential equation with respect to the parameter $z\in \C^*$:
\beq
\label{frob_eq2}
(z\d_z +\E)J_{\A}({\bf t},z) =  \Theta\, J_\A({\bf t},z),
\eeq
where
\ben
\E=\sum_{i=0}^{\mu-1}d_it_i \d_i,\quad (d_i:={\rm deg}\, t_i={\rm deg}\, s_i),
\een
is the {\em Euler vector field} and $\Theta$ is the so-called {\em Hodge grading operator }. The latter is defined by
\ben
\Theta:\T^*_S\rightarrow \T^*_S,\quad \Theta(dt_i)=\Big(1-\frac{D}{2}-d_i\Big)dt_i,
\een
where $D$ is the so called {\em conformal dimension} of the Frobenius manifold, uniquely determined by the symmetry of the degree spectrum: the numbers $d_i$ are
symmetric with respect to the point $1-D/2$.
The compatibility of the system \eqref{frob_eq1}--\eqref{frob_eq2} implies that the residue pairing, the multiplication, and the Euler vector field give rise to a {\em conformal Frobenius structure} of   conformal  dimension $D$. We refer to \cite{D,M} for the definition and more details on Frobenius structures.
\begin{theorem}
[\cite{He, SaT}] Let $f$ be an isolated singularity, a primitive form $\omega$ induces a germ of Frobenius manifold structures $(T_{\s}\S, \LD,\RD, \bullet_{\s}, \E, \d_0)$ with an Euler vector $\E$ and a flat identity $\d_0$ for any $\s\in\S$. It is homogeneous and generically semisimple.
\end{theorem}


\subsection{Global Frobenius manifold structures for simple elliptic singularities.}
Simple elliptic singularities are classified by K.Saito (cf. \cite{Sa2}) into three different types, $\widetilde{E}_6, \widetilde{E}_7,\widetilde{E}_8$. In this paper, we consider three families of simple elliptic singularities by choosing a particular normal form in each family, see \eqref{polynomial} below. The differential equations for the primitive forms will be the same under our choice, see \eqref{eq:Picard-Fuchs}. Besides, all the possible elliptic orbifolds $\mathbb{P}^1$ with three singular points can be seen as mirrors of our families at infinity of the complex plane, referred to as {large complex structure limit point}. The method also works for other normal forms although the mirrors of those elliptic orbifolds may appear in singular points on the complex plane other than the large complex structure limit points. For the framework of {\em global mirror symmetry}, see \cite{CR2}. Such global mirror symmetry phenomena for simple elliptic singularities are studied in \cite{MS}. However, our choices here are enough for describing the quasi-modularity properties of CohFTs for those elliptic orbifolds. Only modular subgroups will be different for different normal forms.  Let $W$ be one of the following three polynomials
\beq\label{polynomial}
\widetilde{E}_6:= x_1^3+x_2^3+x_3^3, \quad \widetilde{E}_7: =x_1^2x_3+x_1x_2^3+x_3^2, \quad
\widetilde{E}_8: =x_1^3x_3+x_2^3+x_3^2.
\eeq
Let us analyze the case $W=\widetilde{E}_6$. The other cases are similar.  We define a 1-dimensional family by
$$W_{\si}=W+\sigma x_1x_2x_3.$$
Note that $W_{\si}$ has an isolated singularity of the same rank $\mu$ iff $\si\in\Sigma$,
\beqa
\Sigma=\Big\{\si\in\mathbb{C}\Big\vert\si^3+27\neq0\Big\}.
\eeqa
Now we can replace $f$ in section 3.1 by $W_{\sigma}$. Its miniversal deformation is
\begin{equation}
F:=W_{\sigma}(\textbf{s},\textbf{x})=W_{\si}+\sum_{i=0}^{\mu-1}s_i\phi_i.
\end{equation}
Here $\{\phi_i\}_{i=0}^{\mu-1}$ is a basis of homogeneous polynomials of the Milnor ring
$\mathscr{Q}_{W_{\si}}$. We always set $\phi_{\mu-1}=x_1x_2x_3, \phi_0=1$ and identify the index $\mu-1$ by $-1$. Thus $\phi_{-1}=x_1x_2x_3$ and $s_{-1}=s_{\mu-1}$.

\subsubsection{Primitive forms and global moduli of Frobenius manifolds}\label{sec:3.2}
Recall that for a generic $(\s,\la)$, the fiber
$$X_{\s,\la}=\big\{\x\in\C^N\big\vert \varphi(\s,\x)=\la\big\}$$
is homotopic to $\mu$ copies of $N-1$ dimensional sphere. The non-generic $(\s,\la)$ are the ones for which $X_{\s,\la}$ has a singularity. They form
an analytic hypersurface called {\em discirminant}. The complement of the latter is a base for the middle cohomology bundle formed by the
middle cohomology groups $H^{N-1}(X_{\s,\la};\C)$. In addition the integral structure in cohomology induces a flat Gauss-Manin connection.

Let us denote by $E_{\sigma}$ be the curve defined by $W_{\sigma}$,
$$E_{\sigma}:=\big\{[x_1,x_2,x_3]\in\C P^2\big\vert W_{\sigma}(x_1,x_2,x_3)=0\big\}.$$
One may compactify the family $X\to\S\times \C$ to $\overline{X}\to\S\times \C$ so that $E_{\sigma}=\overline{X}-X$ is the boundary. $E_{\sigma}$ is also known as the {\em elliptic curve at infinity}, cf.\cite{L}.
According to K. Saito (see \cite{S}), the primitive forms for simple elliptic singularity $W_{\sigma}$ are homogenous of degree 0 and can be expressed as
$$\omega=\frac{d^3{\bf x}}{\pi_{A}(\sigma)}.$$
They are parametrized by the periods of $E_\sigma$,
\beq\label{period}
\pi_A(\si):= 2\pi i \int_{A_\sigma} {\rm Res}_{E_\si}[d^3\x/dF]\ ,
\eeq
where we fix a reference point $\sigma_{0}\in\Sigma$, $A\in H_1(E_{\si_0},\C)$ is some fixed non-zero 1-cycle and $A_\si$ is a flat family of cycles uniquely determined by $A$ for all $\si$ in a small neighborhood of $\si_0$. $d^3\x/dF$ is a holomorphic 2-form on $X_{\si,\la}$ and ${\rm Res}$ is the residue along $E_{\sigma}$. The boundary of any tubular neighborhood of $E_{\si}$ in $\overline{X}_{\s,\la}$ is a circle bundle over $E_{\si}$ that induces via pullback an injective {\em tube map} $L:H_1(E_{\sigma})\to H_2(X_{\s,\la})$. Let $\alpha=L(A);$ then we have
\beq\label{lerey_period}
\pi_A(\si) = \int_{\alpha}\frac{d^3\x}{dF}.
\eeq
We refer to $\alpha$ as a {\em tube} or {\em toroidal} cycle. The space of
all toroidal cycles coincides with the kernel of the intersection
pairing on  $H_2(X_{\s,\la};\C)$.

The space of all periods $\pi_A(\si)$ coincides with the space of solutions of the following differential equation (see the Appendix in \cite{MR}),
\beq\label{eq:Picard-Fuchs}
\frac{d^2}{d\si^2} + \frac{3\si^2}{\si^3+27}\ \frac{d}{d\si} +  \frac{\si}{\si^3+27}= 0.
\eeq
Take $\la=-\si^3/27$, equation \eqref{eq:Picard-Fuchs} is just a Gauss hypergeometric equation,
\beq\label{eq:Gauss}
\la(1-\la)\frac{d^2}{d\la^2}+(\frac{2}{3}-\frac{5\la}{3})\frac{d}{d\la}-\frac{1}{9}=0
\eeq

Now let us describe the global Frobenius manifold structure for those normal forms.
We fix a symplectic basis $\{A',B'\}$ of $H_1(E_{\si_0};\Z)$ once and for all. Then the primitive form is a multi-valued function on $\Sigma$. Thus it is
more natural to replace $\Sigma$ by its universal cover. The latter is
naturally identified with the upper half-plane $\mathbb{H}.$ The points in the universal
cover $\widetilde{\Sigma}$ of $\Sigma$ are pairs consisting of a point
$\sigma\in \Si$ and a homotopy class of paths $l(s)\in\Si$ with $l(0)=\si_0, l(1)=\sigma.$ The map
$$(\si,l(s))\mapsto \tau=\frac{\pi_{B'}(\si)}{\pi_{A'}(\si)},$$
where the periods $\pi_{B'}$ and $\pi_{A'}$ are analytically continued along the path $l(s)$, defines an analytic isomorphism between the universal cover of $\widetilde{\Si}$ and the upper half-plane $\mathbb{H}$.

Let $\M=\mathbb{H}\times\C^{\mu-1}$.
A global Frobenius structure exists on $\M$ for any non-zero cycle
\beq\label{fstr:cycle}
A=dA'+cB'\in H_1(E_{\si_0};\C),\quad -d/c\notin\mathbb{H}.
\eeq
Now let us describe the choice of a coordinate on $\HH$, which we use through out the paper.
Let $M$ be the classical monodromy operator on the middle homology bundle. By definition, $M$ is the linear operator induced by the parallel transport with respect to the Gauss-Manin connection along a loop in $\C^*\equiv \{\si_0\}\times (\C\backslash\{0\})$ based at $\lambda=1$. The operator $M$ is diagonalizable and one can find an eigenbasis $\{\alpha_i\}_{i=-1}^{\mu-2}, \alpha_{i}\in H^*(X_{\sigma,1},\C)$, s.t., the eigenvalue of $\alpha_i$ is $e^{2\pi i d_i}$. Here
$$(\si,1):=(\si,0,\cdots,0,1)\in\S\times\C.$$

In particular, the invariant subspace of $M$ is spanned by $\alpha_{-1}$ and $\alpha_{0}$.  Put $\alpha_0=-(-2\pi)^{3/2}L(A)$ and $\alpha_{-1}=-(-2\pi)^{3/2}L(B)$, where the cycle $B=bA'+aB'$ is chosen to be any cycle linearly independent from $A.$ Then it was proved in \cite{MR} that the function
\beq\label{flat-deg0}
t:=\frac{\pi_{B}(\si)}{\pi_{A}(\si)} = \frac{a\tau+b}{c\tau+d}
\eeq
is a flat coordinate. Slightly abusing the notation we simply write $t\in \HH$ instead of saying that $t$ is given by formula
\eqref{flat-deg0}  for some $\tau\in \HH$.
The entire flat coordinate system can be described in a similar way (see Section 2.2.2 in \cite{MR}).
Hence, $\M$ is a moduli space of global Frobenius manifold structures. For convenience, we denote the flat coordinates by $\t=(t,\t_{\geq0})\in\M$, with $$\t_{\geq0}=(t_0,\cdots,t_{\mu-2})\in\C^{\mu-1}.$$

\subsection{The action of the monodromy group on flat coordinates}\label{sec:3.3}
The monodromy group $\Gamma$ acts on $\M$ by covering
transformations. In this subsection, we recall
its action on flat coordinates.

Let $\nu$ be a \emph{monodromy transformation} in the vanishing homology along a given loop $C$ in $\Sigma$ based at $\sigma_0$. According to \cite{MR}, the $\nu$ action on $\{\ga_i\}_{i=-1}^{\mu-2}$ has a matrix form with respect to the vector of basis $(\alpha_{-1},\cdots,\alpha_{\mu-2})^T$,
\beqa\label{monodromy_matrix}
g\oplus{\rm Diag}(e^{2\pi i d_1k},\dots,e^{2\pi i d_{\mu-2}k}) \in {\rm SL}(2;\C)\times \Z^{\mu-2},
\eeqa
where
\beqa\label{nu}
g(\ga_{-1})=n_{11}\ga_{-1}+n_{12}\ga_0,\quad\mbox{and}\quad
g(\ga_{0})=n_{21}\ga_{-1}+n_{22}\ga_0,
\eeqa
and the matrix $(n_{ij})\in{\rm SL}(2;\C)$.

From now on we fix a flat coordinate system $t_a=t_a(\s)$ $(-1\leq
a\leq \mu-2)$, multi-valued on $\S$ and holomorphic on the cover $\M$, and denote by $H$ the space of flat
vector fields on $\M$. We further assume that the flat coordinates are
chosen in such a way that the residue pairing assumes the form:
\ben
(\d_i,\d_j)=\delta_{i,j'},\quad -1\leq i,j\leq \mu-2,
\een
where $\d_a:=\d/\d t_a$ and $'$ is the involution defined by
$$-1\mapsto 0, \quad 0\mapsto -1,\quad
i\mapsto \mu-1-i, \quad 1\leq i\leq \mu-2.$$

According to \cite{MR} the flat coordinates can be expressed via certain period
integrals as rational functions on the vanishing homology. It follows
that the monodromy group $\Gamma$ acts on the flat coordinates as well
and that this action coincides with the analytic continuation along $C$.
According to \cite{MR} if the flat coordinate system is such that the
residue pairing has the above form; then the monodromy transformation (or
equivalently the analytic continuation) of the flat coordinates
has the following form. Put
\beq
j_{\nu}(t):=j(g,t):=n_{21}t+n_{22};
\eeq
then
\beq
\nu(\t)_{-1}=g(t):=\frac{n_{11}t+n_{12}}{n_{21}t+n_{22}}
\eeq
and
\beq\label{eq:modular}
\nu(\textbf{t})_{0}= t_0+\frac{n_{12}}{2j_{\nu}(t)}\sum_{i=1}^{\mu-2} t_it_{i'},\ \nu(\textbf{t})_{i}=\frac{e^{2\pi i d_ik}}{j_{\nu}(t)}\ t_i,\ 1\leq i\leq \mu-2.
\eeq

\section{Global B-model CohFT  and anti-holomorphic completion}\label{sec:4}

The core of our paper is global B-model CohFTs for simple elliptic singularities, which we will
construct in this section. The basic idea is that the global higher
genus B-model theory of \cite{MR} can be enhanced to a global B-model
CohFT using the construction of Teleman (see section two). The
modularity will follow essentially from the monodromy calculations in
\cite{MR}.

\subsection{Global B-model CohFT}

\subsubsection{Givental's  semisimple quantization operator}

Suppose that $W$ is one of the three families of simple elliptic
singularities under consideration. Recall the global Frobenius
manifold structures on $\M.$ First we
recall the definiton of Givental's quantization operator and then
we use it to define a CohFT  $\La^W(\t)$ over the semisimple loci $\M_{ss}.$

Let $\mathcal{K}\subset \M$ be the set of points $\t$ such
that $u_i(\s(\t))=u_j(\s(\t))$ for some $i\neq j$. We call this set the {\em
  caustic} and put $\M_{ss}$ for its complement. Note that the points
$\t\in \M_{ss}$ are semisimple, i.e., the critical values $u_i(\s(\t))$ ($1\leq i\leq \mu$) form a coordinate system locally near $\t$.
Let $\t\in \M_{ss}$; then we have an isomorphism
\beqa
\Psi(\t): \C^\mu\to T_{\t}\M,\quad e_i\mapsto \sqrt{\Delta_i(\s(\t))}\,\frac{\d}{\d u_i(\s(\t))},
\eeqa
where $\Delta_i(\s(\t))$ is defined by
\beqa
\Big(\frac{\d}{\d u_i(\s(\t))},\frac{\d}{\d u_j(\s(\t))}\Big)=\frac{\delta_{ij}}{\Delta_i(\s(\t))},
\eeqa
and we identify $T_\t\M$ with $H$ via the flat metric, i.e.,
\ben
\frac{\d}{\d u_i}=\sum_{j=0}^{\mu-1} \, \frac{\d t_j}{\d u_i} \, \d_j, \quad 1\leq i\leq\mu.
\een
$\Psi_{\t}$ diagonalizes the Frobenius multiplication and the residue pairing:
\beqa
e_i\bullet e_j =\delta_{i,j}\sqrt{\Delta_i(\s(\t))}\ e_i,\quad (e_i,e_j)=\delta_{ij}.
\eeqa

The system of differential equations (\ref{frob_eq1}) and (\ref{frob_eq2}) admits a unique formal solution of the type
\ben
\Psi(\t) R(\t)\,e^{U(\t)/z},\quad R(\t)={\rm Id}+\sum_{k=1}^{\infty}R_{k}(\t)z^{k}\in {\rm End}(\C^\mu)[[z]].
\een
where $U(\t)$ is a diagonal matrix with entries $u_1(\s(\t)),\dots,u_\mu(\s(\t))$ on the diagonal, cf.\cite{D,G}.

\subsubsection{Global B-model CohFT}

Givental used $R(\t)$ to define a higher genus generating function over $\M_{ss}$. We would like to enhance his definition to CohFT. The main difficulty is to extend our definition to non-semisimple points in $\mathcal{K}.$

For any semisimple point $\t\in\M_{ss}$, we define a CohFT with a flat identity and a state
space $H$ (see  Sect. \ref{sec:2})
\beq\label{eq:trans=quant}
\La^{W}(\t):=\Psi(\t) \circ T_{z}\circ\widehat{R}(\t)\circ T_{z}^{-1}\circ I^{\mu,\Delta(\t)}.
\eeq

We are interested in the loci of points $\t=(t,0)\in\HH\times\C^{\mu-1}$, which are never
semisimple. To continue
our B-model discussion, we need to prove that $\La^W(\t)$ extends
holomorphically for all $\t\in \M.$
To begin with, let us fix $g$, $n$, and $\gamma_i \in H$; for convenience, we denote by
$$\La_{g,n}^{W}(\t):=(\La^W(\t))_{g,n}.$$
$\La^W_{g,n}(\t)(\gamma_1,\dots,\gamma_n)$ is a linear
combination of cohomology classes on $\overline{\M}_{g,n}$ whose
coefficients are functions on $\M$.
\begin{lemma}\label{finite-order-poles}
The coefficients of $\La^W_{g,n}(\t)(\gamma_1,\dots,\gamma_n)$ are
meromorphic functions on $\M$ with at most finite order poles along
the caustic $\mathcal{K}$.
\end{lemma}
\proof
By definition, the CohFT \eqref{eq:trans=quant} depends only on the
choice of a canonical coordinate system $u(\t):=(u_1(\s(\t),\dots,u_\mu(\s(\t))$.
The latter is uniquely determined up to permutation. Note that \eqref{eq:trans=quant}
is permutation-invariant, i.e., it does not matter how we order the canonical
coordinates. On the other hand, up to a permutation $u(\t)$  is invariant under the analytical
continuation along a closed loop in $\M_{ss}.$ It follows that
$\La^W_{g,n}(\t)$ is a single valued function on $\M_{ss}.$

We need only to prove that the poles along $\mathcal{K}$ have finite order.
Note that according to the definition of the class \eqref{do:CohFT}
only finitely many graphs $\Gamma$ contribute. The reason for this is
that in order to have a non-zero contribution, we must have
\beqa
\sum_{e\in E_v(\Gamma)} \, k(e) \ \leq \ 3g_v-3 + r_v + n_v.
\eeqa
Summing up these inequalities, we get
\beqa
\sum_v\sum_{e\in E_v(\Gamma)} \, k(e)\, \leq 3(g-1) - 3{\rm Card}(E(\Gamma)) + \sum_v r_v
+ n,
\eeqa
However $$\sum_v r_v=2{\rm Card}(E(\Gamma)),$$
which implies that the number of edges of $\Gamma$ is bounded by
$3g-3+n$. This proves that there are finitely many possibilities for
$\Gamma$. Moreover, there are only finitely many possibilities for
$k(e)$, i.e., our class is a rational function on the entries of only
finitely many $R_k$. Since each $R_k$ has only a finite order pole
along the caustic the Lemma follows.
\qed

We will prove below that $\La_{g,n}^W(\t)$ is convergent near the point $(\sqrt{-1}\,\infty,0)\in\overline{\HH}\times\C^{\mu-1}$ and that it extends holomorphically through the caustic
(see Theorem \ref{convergence} and Proposition \ref{extension}). Thus $\La^{W}(\t)$ is a CohFT for all $\t\in\M$. In particular,
\beq
\La_{g,n}^{W}(t)=\lim_{\t\in \M_{ss} \rightarrow (t,0)}\La_{g,n}^W(\t)
\eeq
for all $t\in \HH=\HH\times \{0\} \subset \M$.

\subsection{Monodromy group action on $\La_{g,n}^{W}(t)$}

Using the residue pairing we identify $T^*\M$ and $T\M$, i.e., $dt_i =
\d_{i'}$. We also identify ${\rm End}(H)$ with the space of
$\mu\times\mu$ matrices via $A\mapsto (A_{ij})$, where the entries
$A_{ij}$ are defined in the standard way, i.e.,
\ben
A (dt_j) = \sum_{i=-1}^{\mu-2} A_{ij}dt_i \,.
\een
Recall the notation from Section \ref{sec:3.3}: a loop $C$ in $\Si$, inducing via the Gauss-Manin connection
a monodromy transformation $\nu$ on vanishing homology and a transformation of the flat coordinates via
analytic continuation $\t\mapsto \nu(\t)$. The latter induces a monodromy transformation of the stationary phase asymptotics,
which was computed in \cite{MR}, Lemma 4.1. In case $W=\widetilde{E}_6,$ let
\beq\label{matrix:M}
M_{\nu}(\t)=\begin{bmatrix}
j_{\nu}(t)^{-1}  & *               &   *        &  *   \\
 0               & j_{\nu}(t)     &   0        &  0   \\
 0               & *               & e^{4\pi i\,k/3}\,I_3  &  0    \\
 0               & *               &   0        & e^{2\pi i\,k/3}\,I_3
\end{bmatrix}
\in{\rm End}(H)[[z]].
\eeq
where
\beqa
M_{-1,j}= -e^{2\pi i d_jk}\,n_{12}\,j^{-1}_{\nu}(t)\,t_{j},\quad 1\leq j\leq 6
\eeqa
and
\beqa
M_{-1,0}= -n_{12}z - \frac{n_{12}^2}{2j_{\nu}(t)}\sum_{i=1}^6 t_it_{i'},\quad  M_{i,0}=n_{12}t_{i'}, \quad 1\leq i\leq 6.
\eeqa
\begin{lemma}[\cite{MR} ]
The analytic continuation along $C$ transforms
\ben
\Psi(\t)R(\t)e^{U(\t)/z} \quad \mbox{ into } \quad \leftexp{T}{M}_\nu(\t)\Psi_{\t}R(\t)e^{U(\t)/z}\, P,
\een
where $P$ is a permutation matrix and $\leftexp{T}{}$ means transposition.
\qed
\end{lemma}

The CohFT constructed by the analytical continuation along $C$ of
$\Lambda^{W}(\t)$ will be denoted by
$$
\Lambda_{g,n}^{W}\big(\nu(\t)\big)\in H^*(\overline{\M}_{g,n},  \C)\otimes (H^{\vee})^{\otimes n}.
$$
Restricting to $\t_{\geq0}=0$, we have
\beqa
\leftexp{T}M_{\nu}(t):=\lim_{\t_{\geq0}\to0}\leftexp{T}M_{\nu}(\t)=j^{-1}_{\nu}(t)\ J_{\nu}(t).
\eeqa
With
\beq\label{Phi:matrix}
J_{\nu}(t):=
\begin{bmatrix}
1&0\\
0& j^2_{\nu}(t)
\end{bmatrix}
\oplus j_{\nu}(t)\, e^{4\pi i\,k/3}\,I_3 \oplus j_{\nu}(t)\, e^{2\pi i\,k/3}\,I_3
\in{\rm End}(H)[[z]].
\eeq

Now let
\beq\label{X:matrix}
X_{\nu,t}(z)=
\begin{bmatrix}
1 & -n_{12}z/j_{\nu}(t) \\
0 & 1
\end{bmatrix}\bigoplus I_6
\in {\rm End}(H)[[z]].
\eeq
\begin{theorem}\label{thm:analytic-cont}
The analytic continuation
transforms the Coh FT  as follows:
\beq\label{eq:ana-cycle}
\Lambda^{W}\big(\nu(t)\big)
=J_{\nu}^{-1}(t)\circ\widehat{X}_{\nu,t}(z)\circ\Lambda^{W}(t).
\eeq
\end{theorem}
\begin{proof}
The calculation in \cite{MR} also works on cycle-valued level.
\end{proof}
Now we give a lemma which is very useful later on.
\begin{lemma}\label{lm:Commutivity}
Let $E(z)\in\L^{(2)}_{+}{\rm GL}(H)$; then it intertwines with $J^{-1}_{\nu}(t)$ by
\beqa
J^{-1}_{\nu}(t)\circ\widehat{E}(z)=\widehat{E}(j^2_{\nu}(t)z)\circ J^{-1}_{\nu}(t).
\eeqa
\end{lemma}
\begin{proof}
From \eqref{Phi:matrix} and the definition of $J^{-1}_{\nu}(t)\circ$, we know that the pairing $\eta$ is scaled by
$j^2_{\nu}(t)$ when applying $J^{-1}_{\nu}(t)\circ$. Thus the quadratic differential action $\widehat{E}(z)$ becomes $\widehat{E}(j^2_{\nu}(t)z)$.
\end{proof}

\subsection{Anti-holomorphic completion and modular transformation.}
Let $\RR$ or $\mathcal{R}$ be a cohomology ring of any fixed Deligne-Mumford moduli space of stable curves of genus $g$ with $n$ marked points, i.e., $\RR=H^*(\overline{\M}_{g,n},\C)$ for some $2g-2+n>0$.
\begin{definition}
We say  that a $\RR$-valued function $f:\HH\to\RR$ is a $\RR$-valued quasi-modular form of weight $m$ with respect to some finite-index subgroup $\Gamma\subset {\rm SL}_2(\Z)$ if there
are $\RR$-valued functions $f_i$, $1\leq i\leq K$, holomorphic on $\mathbb{H},$ such
that
\begin{enumerate}
\item The functions $f_0:=f$ and $f_i$ are holomorphic near cusp $\tau=i\infty$.
\item
The following $\RR$-valued function
\beqa
f(\tau, \bar{\tau}) = f_0(\tau) +
f_1(\tau)(\tau-\overline{\tau})^{-1}+\dots
+f_K(\tau)(\tau-\overline{\tau})^{-K}.
\eeqa
is modular, i.e., there exists some $m\in\mathbb{N}$ such that for any $g\in \Gamma$,
\ben
f(g\tau, g\overline{\tau}) = j(g,\tau)^m f(\tau, \overline{\tau}).
\een
\end{enumerate}
\end{definition}
$f(\tau, \overline{\tau})$ is called the {\em anti-holomorphic completion} of $f(\tau)$.

\subsubsection{Anti-holomorphic completion of $\Lambda_{g,n}^{W}(t)$}
Let $W$ be the homogeneous polynomial as in \eqref{polynomial}. Denote by
\beq\label{eq:anti-holo-op}
X_{t,\bar{t}}(z) =
\begin{bmatrix}
1 & -z(t-\bar{t})^{-1} \\
&\\
0 & 1
\end{bmatrix}
\oplus I_6
\in{\rm End}(H)[[z]],
\eeq
where $\bar{t}$ is the anti-holomorphic coordinate on $\HH$ defined by (cf. formula \eqref{flat-deg0})
\ben
\bar{t}:=\frac{a\overline{\tau}+b}{c\overline{\tau}+d}\ .
\een
We define the \emph{anti-holomorphic completion} of Coh FT $\Lambda^{W}(t)$ by:
\beq\label{eq:anti-holo}
\Lambda^{W}(t,\bar{t})
:=\widehat{X}_{t,\bar{t}}(z)\circ\Lambda^{W}(t).
\eeq
\begin{theorem}
Under the assumption of extension property, the analytic continuation
of the anti-holomorphic completion $\Lambda_{g,n}^{W}(t,\bar{t})$
along $\nu$ is
$$
J_{\nu}^{-1}(t)\circ\Lambda_{g,n}^{W}\big(t,\bar{t}\big).
$$
\end{theorem}
\begin{proof}
We define an operator $\widehat{X}_{\nu,t,\bar{t}}(z)$, s.t., the
following diagram is commutative:
\beqa
\begin{CD}
&&\La^{W}(t)&@>\widehat{X}_{t,\bar{t}}(z)>> &\La^{W}(t,\bar{t})  \\
& &@VV J_{\nu}(t)\circ\widehat{X}_{\nu,t}(z)V   &   &  @VV\widehat{X}_{\nu,t,\bar{t}}(z)V  \\
   &                    & \Lambda^{W}\big(\nu(t)\big)    &
 @>\widehat{X}_{\nu(t),\nu(\bar{t})}(z)>>   & \Lambda^{W}\big(\nu(t),\nu(\bar{t})\big) \\
   \end{CD}
\eeqa
We need to prove that
\beqa
\widehat{X}_{\nu,t,\bar{t}}(z)=J^{-1}_{\nu}(t).
\eeqa
Let us consider the analytic continuation for $X_{t,\bar{t}}(z)$.
Analytic continuation acts on $(t-\bar{t})^{-1}$ by
\beqa\label{eq:ana-inv}
\frac{1}{\nu(t)-\nu(\bar{t})}=-\Big(\frac{n_{12}}{j_{\nu}(t)}+\frac{1}{t-\bar{t}}\Big)\,j^2_{\nu}(t).
\eeqa
By definition \eqref{eq:anti-holo-op}, this implies
\beq\label{eq:commutivity}
X_{\nu(t),\nu(\bar{t})}(z)
=X_{t,\bar{t}}(j^2_{\nu}(t)z)\,X^{-1}_{\nu,t}(j^2_{\nu}(t)z).
\eeq
Recalling Lemma \ref{lm:Commutivity}, we get,
\beq\label{eq:change-metric}
J^{-1}_{\nu}(t)\circ\widehat{X}_{\nu,t}(z)\circ\widehat{X}_{t,\bar{t}}^{-1}(z)
=\widehat{X}_{\nu,t}(j^2_{\nu}(t)z)\circ\widehat{X}_{t,\bar{t}}^{-1}(j^2_{\nu}(t)z)\circ J^{-1}_{\nu}(t).
\eeq
Thus the result follows from \eqref{eq:commutivity} and \eqref{eq:change-metric}.
\end{proof}

\subsubsection{Cycle-valued quasi-modular forms from $\Lambda_{g,n}^{W}(t)$}
We consider a pair
$$(\vec{\gamma}_I,\iota_{I})=\big((\gamma_1,\cdots,\gamma_n),(\iota_{1},\cdots,\iota_{n})\big)\in H^{\otimes n}\times\Z_{\geq0}^{n}$$
where each $\gamma_i\in\mathscr{S}=\{\d_{-1}=\d_{\mu-1},\d_0,\cdots,\d_{\mu-2}\}$. $I$ is a multi-index
$$I=(i_{-1},i_{0},\cdots,i_{\mu-2})\in\Z_{\geq0}^{\mu},\quad i_{-1}+\cdots+i_{\mu-2}=n.$$
$i_{j}$ is the number of $i\in\{1,\cdots,n\}$ such that $\gamma_i=\d_j$.
Under the assumption of extension property, we define a cycle-valued function $f^{W}_{I,\iota_I}$ on $\HH$,
\beq
f^{W}_{I,\iota_I}(t)
=\La_{g,n}^{W}(t)(\vec{\gamma}_I)\
\in H^{*}(\overline{\M}_{g,n},\C).
\eeq
and its anti-holomorphic completion
$$f^{W}_{I,\iota_I}(t,\bar{t}):=\La_{g,n}^{W}(t,\bar{t})(\vec{\gamma}_I)
.$$
For $\iota_I=(0,\cdots,0)$, we simply denote them by $f^{W}_{I}(t)$ and $f^{W}_{I}(t,\bar{t})$. Let
\beq\label{weight}
m(I):=2i_{-1}+\sum_{j=1}^{\mu-2}i_{j}.
\eeq

\begin{theorem}\label{qm:transf}
Let $W$ be a simple elliptic singularity. Then $f^{W}_{I,\iota_I}(t)$ satisfies the transformation law of cycle-valued quasi-modular forms of weight $m(I)$.
\end{theorem}
\begin{proof}
First we consider $\iota_I=(0,\cdots,0)$. It is easy to see $f^{W}_I(t,\bar{t})$ is an anti-holomorphic completion for $f^{W}_I(t)$ and for monodromy $\nu$ described as before, we have
\ben
f^{W}_I(\nu (t),\nu (\bar{t}))
&=&\big(\widehat{X}_{\nu,t,\bar{t}}(z)\circ\La^{W}(t,\bar{t})\big)_{g,n}(\vec{\gamma}_I)\\
&=&j_{\nu}^{m(I)}(t)\,\La^{W}_{g,n}(t,\bar{t})(\vec{\gamma}_I)\\
&=&j_{\nu}^{m(I)}(t)\,f^{W}_I(t,\bar{t}).
\een
Now the statement follows from monodromy acts trivially on psi-classes.
\end{proof}
\begin{remark}
For $f^W_I(t)$ to be a cycle-valued modular form, it needs to be holomophic at $\tau=\sqrt{-1}\, \infty$ (cf. formula \eqref{flat-deg0}). This will be achieved by the mirror theorem in section 5. Hence,  by combining A-model with B-model,
we produce cycle-valued quasi-modular forms.
\end{remark}

\section{A-model CohFT and cycle valued modular forms}\label{sec:5}

In the last section we constructed an anti-holomorphic modification of  the  B-model CohFT, such that it has
the correct transformation property under analytic continuation. However, we still have to prove the following
two properties: (1) the CohFT extends holomorphically through the caustic; (2) the quasi-modular forms are holomorphic
at the cusp $\tau=\sqrt{-1}\, \infty$. We address both issues using an A-model (Gromov-Witten) CohFT and
mirror symmetry. As a byproduct we obtain a geometric interpretation of the B-model CohFT
as the Gromov-Witten CohFT of an elliptic orbifold $\P^1$ and we obtain a proof of our main result Theorem \ref{t1}.

The hard part of the argument is already completed in \cite{KS, MR}. Our goal is to recall the appropriate results and to show
in what order they have to be used. The idea is as follows. We first establish analyticity and generic semisimplicity of the
genus zero  Gromov-Witten theory. This is done by using an estimate for the GW invariants and genus zero mirror symmetry. Then,
we make use of a result of Coates-Iritani in order to prove the convergence of the Gromov-Witten ancestor CohFT of all genera.
The last major step is a higher genus mirror symmetry that allows us to match the Gromov-Witten ancestor CohFT with the B-model
CohFT near the large complex limit. This implies the extension property at $\tau=\sqrt{-1}\, \infty$. Finally, we use   Lemma 3.2 from \cite{MR} to conclude the extension property over entire B-model moduli space $\M$.

\subsection{A-model}
Let us recall a general mirror symmetry construction, called {\em Berglund-H\"ubsch-Krawitz} mirror symmetry. For a quasi-homogeneous polynomial $W$ with a suitable symmetry group $G$, a pair of mirror $(W^T,G^T)$ is constructed, \cite{BH,K}. In our case, we choose a cubic polynomial $W^T$ with the maximal admissible group $G^T=G_{W^T}$, and consider this pair in A-model side. Its mirror will be the pair $(W,G=\{{\rm Id}\})$. So the B-model will be Saito-Givental's theory on the miniversal deformation of the family $W_{\sigma}$.

For $W=\widetilde{E}_i, i=6,7,8$  (see \eqref{polynomial}) the mirror $W^T$ is given respectively
by the following cubic polynomials:
\beq
W^T=x_1^3+x_2^3+x_3^3,\quad x_1^2x_2+x_2^3+x_1x_3^2,\quad x_1^3+x_2^3+x_1x_3^2.
\eeq
The weights are $q_i=1/3$, for all $i=1,2,3.$ Consider a hypersurface in the projective space,
$$X_{W^T}=\{(x_1,x_2,x_3)\vert W^T(x_1,x_2,x_3)=0\}\hookrightarrow\P^2.$$
Its {\em maximal admissible group} is
\beqa
G_{W^T}:=\big\{(\la_1,\la_2,\la_3)\in\mathbb{C}^3\big\vert\,W(\la_1\,x_1,\la_2\,x_2,\la_3\,x_3)=W^T(x_1,x_2,x_3)\big\}.
\eeqa
It contains a subgroup $\LD J\RD$, generated by the {\em exponential grading element}
$$J:=(\exp(2\pi i\cdot q_1),\exp(2\pi i\cdot q_2),\exp(2\pi i\cdot q_3))\in G_{W^T}.$$
$\LD J\RD$ acts trivially on $X_{W^T}$. We denote by
$$\widetilde{G}=G/\LD J\RD.$$
The quotient space
$$\X_{W^T}:=X_{W^T}/\widetilde{G_{W^T}}$$
is an elliptic orbifold with $\P^1$ as its underlying space. The A-model is the orbifold Gromov-Witten theory of $\X:=\X_{W^T}$.

\subsection{Analyticity and generic semisimplicity}
Let $H$ be the Chen-Ruan cohomology of $\X$ with unit ${\bf 1}$ and Poincar\'e pairing $\eta$. Let the divisor $\D$ be a nef generator in $H^2(\X,\Z)\subset H^2_{\rm CR}(\X,\Z)$ and
let $\t=(t,t_0,t_1\dots,t_{\mu-2})$ be a linear coordinate system on $H$, such that $t$ is the coordinate along $\D$. Recall the Gromov-Witten CohFT ${}_{\t}\Lambda^{\X}$, which a priori is only formal. Due to the so called  {\rm divisor axiom} we can identify
$q=e^t$, i.e.,
$$
{}_{\t}\Lambda^{\X}_{g,n}(\gamma_1, \cdots,\gamma_n)\in
H^*(\overline{\M}_{g,n}, \C)\otimes\C[[e^t,t_0,\cdots,t_{\mu-2}]].
$$
For every $\alpha,\beta,\gamma\in H$, the {\em big quantum product} $\alpha\star_{\t} \beta$ is defined by  the relation
\beq\label{Frob-algebra}
\langle \alpha\star_{\t} \beta, \gamma\rangle={}_{\t}\Lambda^{\X}_{0,3}(\alpha,\beta,\gamma).
\eeq
The product is only formal in $\t$.
We would like to prove that $\star_{\t}$  is convergent in the
open polydisk $D_{\epsilon}\subset\C^{\mu}$ with center the origin and radius $\epsilon$, i.e, $(q=e^t,\t_{\geq0})\in D_{\epsilon}$.
More precisely, our goal is to prove the following theorem.
\begin{theorem}\label{formal-ss}
The following statements hold:
\begin{itemize}
\item[(1)] There exists an $\epsilon >0$ such that ${}_{\t}\Lambda^{\X}_{0,3}(\alpha,\beta,\gamma)$ is convergent for all $(q=e^t,\t_{\geq0})\in D_\epsilon$ and $\alpha,\beta,\gamma\in H$.
\item[(2)] The quantum product $\star_{\t}$ is generically semisimple.
\end{itemize}
\end{theorem}
Part (1) follows from Theorem 1.2 in \cite{KS}. For the reader's convenience, we sketch the proof here as well.
First let us denote by
\begin{equation}
I_{0,n,d}^{GW}:=\max_{-1\leq i_j\leq\mu-2}\big|\LD\d_{i_1},\cdots,\d_{i_n}\RD_{0,n,d}^{\X}\big|.
\end{equation}
By direct computation, $I_{0,3,0}^{GW}\leq1$.
\begin{lemma}[Lemma 4.16 in \cite{KS}]
For $n+d\geq4,$ we have:
\begin{equation*}
I_{0,n,d}^{GW}\leq
\left\{
\begin{array}{ll}
d^{n-5}C^{n+d-4}, &d\geq1. \\
C^{n-4},&d=0.
\end{array}\right.
\end{equation*}
Here $C$ is some positive constant depending only on $n$.
\end{lemma}

Since $H^*(\overline{\M}_{0,3}, \C)\cong\C$, it is enough to prove the convergence of the corresponding ancestor Gromov-Witten invariants. The divisor axiom implies
$$
\int_{\overline{\M}_{0,3}}{}_{\t}\Lambda^{\X}_{0,3}(\alpha,\beta,\gamma)=
\sum_{d\geq0}q^d\,\sum_{k=0}^{\infty}
\sum_{k_0+\cdots+k_{\mu-2}=k}
\frac{\LD\alpha,\beta,\gamma,\cdots\RD^{\X}_{0,3+k,d}}{k_0!\cdots k_{\mu-2}!}
\prod_{0\leq i\leq\mu-2}t_{i}^{k_i}
$$
where the dots stand for the insertion $\d_0,\cdots,\d_{\mu-2}$ with multiplicities respectively $k_0,\dots,k_{\mu-2}$. For dimensional reasons the Gromov-Witten invariants in the above formula vanish except for finitely many $k$. In another words,
$$
\int_{\overline{\M}_{0,3}}
{}_{\t}\Lambda^{\X}_{0,3}(\alpha,\beta,\gamma)\in
\C[t_0,\cdots,t_{\mu-2}]\otimes\C[[q]].
$$
Thus the convergence of ${}_{\t}\Lambda^{\X}_{0,3}(\alpha,\beta,\gamma)$ in $(q,\t_{\geq0})\in D_{\epsilon}$ follows from the convergence of the following series near $q=0$,
$$\sum_{d\geq0}q^d\,d^{n-5}C^{n+d-4}.$$

Part (2) is not so easy to prove directly in the settings of Gromov-Witten theory.
We use the genus-0 part of the mirror symmetry theorem of \cite{KS}.
We recall the {\em genus-$0$ ancestor Gromov-Witten potential} constructed from ${}_{\t}\La^{\X}$
\beqa
\mathcal{F}_{0}^{GW}(\X)(\t):=\sum_{n}\sum_{\iota_j,i_j,d}\frac{1}{n!}\LD\tau_{\iota_{1}}(\d_{i_1}),\cdots,\tau_{\iota_{n}}(\d_{i_n})\RD^{\X}_{0,n,d}(\t)\prod_{j=1}^{n}\widetilde{q}_{i_j}^{\iota_j}.
\eeqa
We can expand it as a formal series (due to the divisor axiom) as follows:
$$
\mathcal{F}_{0}^{GW}(\X)(\t)\in \C[[q=e^t,t_0,\cdots,t_{\mu-2}]]\otimes \C[[\q_0,\widetilde{\q}_1,\q_2,\cdots]].
$$
The space $\M=\HH\times \C^{\mu-1}$ can be equipped with flat coordinates $\t^B$ corresponding to a  choice of a primitve form (for $W$). In fact, $\M$ has a generically semi-simple Frobenius structure, which allows us to define the Saito--Givental ancestor potentials $\mathcal{F}_{g}^{SG}(W)(\t^B)$ for all genera $g$ (see \ref{formal-potential} ).

The genus-0 mirror symmetry can be stated as follows:
the primitive form can be chosen in such a way that there exists an analytic embedding $D_\epsilon \hookrightarrow \M$, called a {\em mirror map}, s.t.,
\begin{itemize}
\item[(1)]
The linear coordinates $\t$ on $D_\epsilon$ correspond to flat coordinates $\t^B$ on $\M.$
\item[(2)]
We have
\beqa\label{genus0-mirror}
\mathcal{F}_{0}^{GW}(\X_{W^T})(\t)=\mathcal{F}_{0}^{SG}(W)(\t^B).
\eeqa
\end{itemize}
We denote the image of $D_{\epsilon}$ by $D_{\epsilon}^B$. Let us recall (see \cite{MR}) that under the mirror map
the modulus $\tau$ (cf. Sect. \ref{sec:3.2}) is a flat coordinate on $\M$ and we have
\beq\label{mirror-map}
t=\frac{2\pi \sqrt{-1} }{N}\, \tau,
\eeq
where $N=3,4$, and $6$ respectively for $W=\widetilde{E}_6, \widetilde{E}_7,$ and $\widetilde{E}_8.$
It follows that the large volume limit point $e^t=0$, i.e., $t=-\infty$ corresponds to the large
complex structure limit point $\tau=\sqrt{-1}\,\infty.$

The proof can be splitted into two parts: choice of a primitive form, s.t., (1) holds and prove that the ancestor potentials on both sides are uniquely determined from a finite set of correlators, which agree under the mirror map. The first step was done in \cite{MR} and the second one in \cite{KS}.

\subsection{Convergence of $\Lambda^{\X}_{g,n}(\t)$}
We identify via the mirror map the flat coordinates $\t^B$ on $\M$ and the linear coordinates $\t$ on $D_\epsilon$. Recall the
CohFT $\La_{g,n}^{W}(\t^B)$ defined by formula \eqref{eq:trans=quant} for all semisimple points $\t^B$.
\begin{theorem}\label{convergence}
The CohFT $\La_{g,n}^{W}(\t^B)$ extends holomorphically for all $\t^B\in D^B_\epsilon$,
the ancestor Gromov--Witten CohFT ${}_{\t}\La^{\X}$ is convergent for all $\t\in D_\epsilon$, and
we have
\ben
{}_{\t}\La_{g,n}^{\X} = \La_{g,n}^{W}(\t^B),\quad \forall t\in D_\epsilon  .
\een
\end{theorem}
\proof
The Frobenius structure of the quantum cohomology is generically semi-simple (cf. Theorem \ref{formal-ss}, (2)).
In particular, if we think of the CohFT ${}_\t \Lambda^\X$ as a CohFT over the field
\ben
\overline{{\rm Frac}\, \C[[e^t,t_{0},\cdots,t_{\mu-2}]] },
\een
where overline means {\em algebraic closure} and Frac stands for the {\em field of fractions}; then  ${}_\t \Lambda^\X$
is a semi-simple CohFT with a flat identity. Teleman's reconstruction Theorem \ref{hg-rec} applies and we get that
\beq\label{eq:mirror-g}
{}_{\t}\La_{g,n}^{\X}=\La_{g,n}^{W}(\t^B),
\eeq
where the equality should be interpreted as equality in the space
\ben
H^*(\overline{\M}_{g,n},\C)\otimes\overline{{\rm Frac}\, \C[[e^t,t_{0},\cdots,t_{\mu-2}]] }.
\een
On the other hand, according to Lemma \eqref{finite-order-poles},
$\La_{g,n}^{W}(\t^B)$ is meromorphic for $\t\in D^B_{\epsilon}$, thus
\beq
\La_{g,n}^{W}(\t^B)={}_{\t}\La^{\X}_{g,n}\in H^*(\overline{\M}_{g,n},\C)\otimes
\overline{{\rm Frac}\, \C\{e^t,t_{0},\cdots,t_{\mu-2}\}},
\eeq
where $\C\{x_1,\dots,x_n\}$ is the ring of convergent power series at $x_1=\cdots=x_n=0$
(the overline means algebraic closure).
On the other hand, by definition
\beq
{}_{\t}\La_{g,n}^{\X}\in H^*(\overline{\M}_{g,n},\C)\otimes\C[[e^t,t_{0},\cdots,t_{\mu-2}]].
\eeq
Now we apply the following lemma of Coates--Iritani,
\begin{lemma}
[\cite{CI}, Lemma 6.6] The intersection
\ben
\overline{{\rm Frac}\,\C\{x_1,\cdots,x_n\}}\cap\C[[x_1,\cdots,x_n]]\subset\overline{{\rm Frac}\,\C[[x_1,\cdots,x_n]]}
\een
coincides with
$\C\{x_1,\cdots,x_n\}.$
\end{lemma}
\noindent
This completes the proof.
\qed

\subsection{Extension property}
In this subsection, we use Lemma 3.2 from \cite{MR} to derive the extension property.
\begin{proposition}\label{extension}
The coefficients of $\La_{g,n}^W(\t^B)(\gamma_1,\dots,\gamma_n)$ extend
holomorphically through $\mathcal{K}$, i.e., they are holomorphic functions on $\M$.
\end{proposition}
\proof
Let us define an action of $\C^*$ on $\M=\HH\times \C^{\mu-1}$ according to the weights of the
coordinates $\t^B$. Since $\La^W(\t^B)$ is a homogeneous CohFT, the domain $\widetilde{\mathcal{K}}$ of
all $\t^B$ where the theory does {\em not} extend analytically is $\C^*$-invariant. Since $\widetilde{\mathcal{K}}$
is the set of points $\t^B\in \M$, such that $\La^W(\t^B)$ has a pole, $\widetilde{\mathcal{K}}$ must be an
analytic subset. Let us assume that $\widetilde{\mathcal{K}}$ is non-empty. The Hartogues extension theorem implies that the codimension of $\widetilde{\mathcal{K}}$ is at most 1
and hence precisely one. On the other hand, according to Theorem \ref{convergence}, the polydisk $D_\epsilon$ is disjoint from
$\widetilde{\mathcal{K}}$. In particular, $\HH\times \{0\}$ is not contained in $\widetilde{\mathcal{K}}$ and hence the
two subvarieties interesect transversely. This combined with the $\C^*$ invariance of $\widetilde{\mathcal{K}}$ implies
that the connected components of $\widetilde{\mathcal{K}}$ have the form $\{\tau_0\}\times \C^{\mu-1}$. This is a
contradiction, because  $\widetilde{\mathcal{K}}\subset \mathcal{K}$, while
$\{\tau_0\}\times \C^{\mu-1}\not\subset \mathcal{K}.$
\qed

\subsection{Quasi-modularity}
Finally, let us complete the proof of our main theorem. According to Theorem \ref{convergence} the
Gromov--Witten CohFT $\La_{g,n}^{\X}(q)$ is convergent and it coincides with $\La_{g,n}^{W}(\tau)$, under the mirror map \eqref{mirror-map}. The
latter transforms as a quasi-modular form according to Theorem \ref{qm:transf}, it is analytic for all
$\tau\in \HH$ due to Proposition \ref{extension}, and finally it extends holomorphically over the cusp $\tau=i\, \infty$ because $\La_{g,n}^{\X}(q)$ extends holomorphically over $q=0$. This completes the proof of Theorem \ref{t1}.

\end{document}